\documentclass[a4paper,12pt]{article}

\newcommand{\deftext}{\textit}

\usepackage{amsmath}
\usepackage{amsfonts}
\usepackage{amsthm}
\usepackage{url}

\newtheorem{thm}{Theorem}

\title{The non-synchronizing primitive groups of degree up to 624}
\author{
Leonard H.~Soicher\\
School of Mathematical Sciences\\
Queen Mary University of London\\
Mile End Road, London E1 4NS, UK\\
\texttt{L.H.Soicher@qmul.ac.uk}
}

\begin{document}

\maketitle

\begin{center}
Dedicated to the memory of Richard Parker
\end{center}
\bigskip

\begin{abstract}
We describe the methods and results of a classification of the
non-synchronizing primitive permutation groups of degree up to $624$.
We make use of theory and computation to determine the primitive
groups of degree up to 624 that are non-separating, which is a
necessary, and very often sufficient, condition for a primitive group
to be non-synchronizing, and determine exactly which of these
non-separating groups are non-synchronizing. A new infinite sequence of
non-synchronizing graphs is discovered and new examples of non-separating,
but synchronizing, groups are found.  
\end{abstract}

\medskip


\medskip


\section{Introduction}

Interest in synchronizing groups arose from the theory of finite automata,
but the study of synchronizing and related classes of finite primitive
permutation groups is now a fascinating area of research in its own
right.  For an excellent survey of the area, see \cite{ACS}.
In particular, the study of synchronizing groups has strong connections
with important problems in combinatorics, finite geometry, and computation.

In this paper we describe the methods and results of a classification of
the non-synchronizing primitive permutation groups of degree up to $624=5^4-1$.
Our approach is to determine which of the primitive groups of degree up
to $624$ are non-separating, which is a necessary, and very often
sufficient, condition for a primitive group to be
non-synchronizing. We then determine
exactly which of these non-separating groups are non-synchronizing. This
leads to new examples of non-separating, but synchronizing, groups.

We make extensive use of both theoretical and computational tools,
including the application of high-performance parallel computing
in certain cases.  As well, our work has motivated the
development of an algorithm for proper vertex-colouring exploiting
graph symmetry \cite{Soi2} and the discovery of new infinite
sequences of non-synchronizing graphs. We hope to inspire new
discoveries.

In \cite{nonsep}, we provide our main \textsf{GAP} \cite{GAP}
program, a log-file of a run of this program (taking about 57 hours
on a single processor), and a data file containing detailed human and
\textsf{GAP} readable information on all the non-separating primitive
groups classified, including explicit permutation generators for each.
We make use of the Primitive Groups Library of the \textsf{GAP} system,
which provides each of the primitive groups $G$ of degree at most 4095
(up to relabelling the points of its permutation domain) as well as
an O'Nan-Scott type for $G$. Other than for certain primitive groups
of affine type of prime-squared or prime-cubed degree, we provide in
\cite{nonsep} a non-separating graph for each non-separating group
classified, and if the group is almost simple and non-synchronizing,
we provide a non-synchronizing graph for the group.  The graphs are
given in \textsf{GRAPE} \cite{GRAPE} format.

\section{Non-separating and non-synchronizing graphs and groups}

In the present work, we avoid the theory of finite automata, and define
non-separating and non-synchronizing permutation groups in terms of the
graphs on which they act as groups of automorphisms (see \cite{ACS} for
the justification). It is worth pointing out, however, that a permutation
group $G$ on a finite set $\Omega$ is synchronizing if and only if, 
for every non-permutation $f:\Omega\to\Omega$, the 
transformation semigroup generated by $G$ and $f$ contains a constant
mapping.

In this paper, all graphs are finite and undirected,
with at least one vertex, but with no loops and no multiple edges.
Our notation for group names follows that of the \textsf{GAP} Primitive
Groups Library.

Let $\Gamma$ be a graph.  A \deftext{clique} of $\Gamma$ is a
set of pairwise adjacent vertices, and a \deftext{coclique} (or
\deftext{independent set}) of $\Gamma$ is a set of vertices, no two of
which are adjacent. The \deftext{clique number} $\omega(\Gamma)$ is
the size of a largest clique of $\Gamma$, and the \deftext{independence
number} $\alpha(\Gamma)$ is the size of a largest coclique. The
\deftext{automorphism group} $\mathrm{Aut}(\Gamma)$ of $\Gamma$ 
consists of all the permutations $g$ of its vertex set 
$V(\Gamma)$ such that
$\{a,b\}$ is an edge of $\Gamma$ if and only if $\{a^g,b^g\}$ is
also an edge. If $\Gamma$ is vertex-transitive (i.e{.} has a group
of automorphisms transitive on $V(\Gamma)$) then
\[ \omega(\Gamma)\alpha(\Gamma)\le |V(\Gamma)|\] (see, for example,
\cite[Theorem~5.3]{ACS}).

A \deftext{proper vertex-colouring} of $\Gamma$ is a labelling of
its vertices by elements from a set of \deftext{colours}, such that
adjacent vertices are labelled with different colours. Where $k$ is a
non-negative integer, a \deftext{vertex $k$-colouring} of $\Gamma$ is
a proper vertex-colouring using at most $k$ colours. A \deftext{minimum
vertex-colouring} of $\Gamma$ is a vertex $k$-colouring with $k$ as small
as possible, and the \deftext{chromatic number} $\chi(\Gamma)$ 
is the number of colours used in a minimum vertex-colouring of $\Gamma$.
Clearly, $\omega(\Gamma)\le \chi(\Gamma)$.

We say that a graph $\Gamma$ is \deftext{non-synchronizing} if $\Gamma$
is neither a null graph (a graph with no edges) nor a complete graph, and
$\omega(\Gamma)=\chi(\Gamma)$. A permutation group $G$ on $\Omega$ is said
to be \deftext{non-synchronizing} if there is a non-synchronizing graph
$\Gamma$ with vertex-set $\Omega$ such that $G\le \mathrm{Aut}(\Gamma)$;
otherwise $G$ is \deftext{synchronizing}.  It is easy to see that if a
non-trivial permutation group $G$ on $\Omega$ is synchronizing, then $G$
must be \deftext{primitive}, that is, $G$ is transitive on $\Omega$,
and $G$ preserves no partition of $\Omega$ other than the ``trivial"
ones having exactly $1$ or $|\Omega|$ parts.

We say that a vertex-transitive graph $\Gamma$ is \deftext{non-separating}
if $\Gamma$ is neither a null graph nor a complete graph and
$\omega(\Gamma)\alpha(\Gamma)=|V(\Gamma)|$. A transitive permutation group
$G$ on $\Omega$ is said to be \deftext{non-separating} if there is a
non-separating graph $\Gamma$ with vertex-set $\Omega$ such that $G\le
\mathrm{Aut}(\Gamma)$; otherwise $G$ is \deftext{separating}.

Now suppose that $\Gamma$ is vertex-transitive and non-synchronizing. Then
\[ \chi(\Gamma)=\omega(\Gamma)\le |V(\Gamma)|/\alpha(\Gamma)\le
\chi(\Gamma),\] so we must have $\chi(\Gamma)=|V(\Gamma)|/\alpha(\Gamma)$,
so the set of colour-classes of a minimum vertex-colouring of $\Gamma$
is a partition of $V(\Gamma)$ by $\omega(\Gamma)$ cocliques of size
$\alpha(\Gamma)$. In particular, a vertex-transitive non-synchronizing
graph must be non-separating, so a separating transitive group must
be synchronizing.

Note that if a transitive permutation group $G$ on $\Omega$ 
is 2-set transitive (in
particular if $G$ is 2-transitive) then any graph $\Gamma$ on $\Omega$
with $G\le \mathrm{Aut}(\Gamma)$ must be null or complete, so 
$G$ is separating and synchronizing.

Note also that if $\Gamma$ is a vertex-transitive graph with a 
prime number of vertices,
such that $|V(\Gamma)|=\omega(\Gamma)\alpha(\Gamma)$, then either
$\omega(\Gamma)=1$ or $\alpha(\Gamma)=1$, so $\Gamma$ is null or complete,
so $\Gamma$ is separating. Thus, any transitive group of prime degree
is separating, and hence synchronizing.

We remark that the separating rank~3 graphs were (mostly) determined
in \cite{BGLR2}. In his PhD thesis \cite{Sch}, Schaefer lists the
non-synchronizing 2-closed primitive permutation groups of degree up to 100
(but wrongly includes the non-separating but synchronizing such groups),
as well as listing the non-synchronizing 2-closed primitive rank~3 groups
of degree 101 to 630.

\section{The Hamming graphs} 

Let $d,m$ be positive integers and let $M$ be a set of size $m$. The
\deftext{Hamming graph} $H(d,m)$ has vertex-set $M^d$, with two vertices
joined by an edge precisely when they differ in just one co-ordinate.

It is known that $H(d,m)$ is non-synchronizing if $d,m>1$.
Here is the argument.  Suppose $M=\{0,\ldots,m-1\}$.  We obtain a clique
of size $m$ in $H(d,m)$ by taking the set of vertices having zeros in
the first $d-1$ co-ordinates.  We obtain a vertex $m$-colouring of
$H(d,m)$ by colouring a vertex with the sum of its
co-ordinates mod $m$.  It follows that $\chi(H(d,m))=\omega(H(d,m))=m$.
Furthermore, if $d,m>1$ then $H(d,m)$ is neither null nor complete.

\section{The O'Nan-Scott Theorem}

The famous O'Nan-Scott Theorem can be used to divide up the 
primitive permutation groups in various useful ways. 
For our purposes, we state a version based on \cite[Chapter~4]{Cam}. 

\begin{thm}[O'Nan-Scott]
Let $G$ be a primitive permutation group on a finite set $\Omega$. Then 
(at least) one of the following holds:
\begin{enumerate}
\item 
$|\Omega|=m^d$ for some integers $d>1,m>2$, 
and $G\le \mathrm{Aut}(H(d,m))$; 
\item $G$ is of \deftext{affine type}, i.e{.}
$\Omega$ can be identified with the vector space
$\mathrm{GF}(p)^d$ (for some prime $p$ and positive integer $d$),
and $T\trianglelefteq G\le\mathrm{AGL}(d,p)$, where $T$ is the
\deftext{translation subgroup} isomorphic to $GF(p)^d$ acting by addition;
\item $G$ is of \deftext{almost simple type}, i.e{.}
$S\le G\le \mathrm{Aut}(S)$ for some non-abelian simple 
group $S$;
\item $G$ is of ``diagonal type", so in particular,
the socle of $G$ consists of the direct product of two or 
more isomorphic copies of a non-abelian simple group.
\end{enumerate}
\end{thm}

Thus, a (non-trivial) synchronizing permutation group must be 
primitive of affine, almost simple, or diagonal type. 
Moreover, we have the following very useful result:

\begin{thm}[Bray, Cai, Cameron, Spiga, Zhang \cite{BCCSZ}] 
Suppose $G$ is a primitive permutation group on a finite set $\Omega$
and $G$ is \textit{not} almost simple. Then $G$ is synchronizing
if and only if $G$ is separating. 
\end{thm}

\section{Some non-synchronizing graphs}

We curate a collection of non-synchronizing graphs, such that every
non-synchronizing, non-affine, non-diagonal primitive 
group of degree at most 624 is
a group of automorphisms of a graph isomorphic to one in our collection.

\subsection{Some infinite sequences of non-synchronizing graphs}

First, we describe some known infinite sequences of non-synchronizing
graphs defined combinatorially.

\begin{thm} 
\label{nonsynchcombi}
The following graphs are non-synchronizing:
\begin{enumerate}
\item where $d,m>1$, the Hamming graph $H(d,m)$;
\item where $1<k<n$ and $k$ divides $n$, the complement $\bar{K}(n,k)$
of the Kneser graph $K(n,k)$ (the \deftext{Kneser graph} $K(n,k)$ has 
vertex-set the $k$-subsets of $\{1,\ldots,n\}$, with two vertices 
joined by an edge precisely when they have trivial intersection);
\item where $n\ge 9$ and $n$ is congruent to 1 or 3 (mod~6),
the graph $J(n,3)$ whose vertices are the 3-subsets of $\{1,\ldots,n\}$,
and with vertices $v,w$ joined by an edge exactly when $|v\cap w|=2$;
\item
where $k>1$, $k$ divides $n$, and $n/k>2$, the graph whose vertices are
the partitions of $\{1,\ldots,n\}$ into $k$-subsets, with two vertices
joined by an edge exactly when they have no part in common.
\end{enumerate}
\end{thm}

\begin{proof}
\begin{enumerate}
\item 
We have already handled this case.
\item
As explained in \cite[Section~6.1.1]{ACS}, this is an application of
Baranyai’s Theorem \cite{Bar}, which says that when $k$ divides $n$,
the $k$-subsets of $N:=\{1,\ldots,n\}$ can be partitioned into classes,
each of which is a partition of $N$. These classes form the colour-classes
of a proper vertex-colouring of $\bar{K}(n,r)$, which has a clique
consisting of all the $k$-subsets containing a chosen element (say $1$)
of $N$. This clique has exactly one vertex in each colour-class, so it
follows that $\chi(\bar{K}(n,k))=\omega(\bar{K}(n,k))$. Furthermore,
since $k$ divides $n$ and $1<k<n$ then $\bar{K}(n,k)$ is non-null and
non-complete.
\item 
This is proved in \cite[Section~6.1.2]{ACS}.
\item
This is a result of Peter Cameron, motivated by the present work.
The only proof which has appeared so far is in his blog
``Cameron Counts" \cite{CameronCounts}, so for the reader's convenience,
we reproduce his proof here, lightly edited.

Suppose $k>1$, $k$ divides $n$, and $n/k>2$. Let $\Gamma$ be the graph
whose vertices are the partitions of $N:=\{1,\ldots,n\}$ into $k$-subsets,
with two vertices joined by an edge exactly when they have no part in
common. First note that since $1<k<n$, $\Gamma$ is not the null graph,
and since $n/k>2$, $\Gamma$ is not the complete graph.

We colour the vertices of $\Gamma$ as follows. Choose an element
$x$ of $N$, and for each $(k-1)$-subset $A$ of $N$ not containing $x$
assign colour $c_A$ to a partition $P$ if the part of $P$ containing $x$
is $\{x\}\cup A$. Each colour class is an independent set, so we obtain
a proper vertex-colouring.

We now apply Baranyai’s Theorem \cite{Bar} to find an appropriate
clique. The $k$-subsets of $N$ can be partitioned into classes, each
of which is a partition of $N$. Of the resulting partitions, clearly no
two share a part, and so they form a clique in the graph. This clique
has exactly one vertex in each colour class above.
It follows that $\chi(\Gamma)=\omega(\Gamma)$.
\end{enumerate} 
\end{proof}

The geometry of certain projective or polar spaces can be used to
produce infinite sequences of non-synchronizing graphs.  We use the
notation $\mathrm{PG}(d,q)$ for the \deftext{projective space} of
(projective) dimension $d$ over the finite field $\mathrm{GF}(q)$,
whose \deftext{points}, \deftext{lines}, $\ldots$, are respectively
the one, two, $\ldots$ dimensional subspaces of the $(d+1)$-dimensional
vector space over $\mathrm{GF}(q)$, with incidence given by symmetrized
inclusion.  The \deftext{symplectic polar space} $\mathrm{W}(2r-1,q)$
is the subgeometry of $\mathrm{PG}(2r-1,q)$ consisting of the points,
lines, $\ldots$ on which a given nondegenerate alternating (bilinear)
form is identically $0$, and the \deftext{Hermitian polar space} ${\cal
H}(d,q^2)$ is the subgeometry of $\mathrm{PG}(d,q^2)$ consisting of
the points, lines, $\ldots$ on which a given nondegenerate Hermitian
(sesquilinear) form is identically $0$.  These geometries are well-studied
(see, for example, \cite[Chapter~4]{Bal}), and can conveniently be
constructed using the \textsf{GAP} package \textsf{FinInG} \cite{FinInG}.

The \deftext{point graph} (resp. \deftext{line graph}) of a geometry
$\cal G$ having points and lines is the graph whose vertices are the
points (resp. lines) of ${\cal G}$, with distinct vertices joined by
an edge precisely when they are incident with a common line (resp.
incident with a common point).

\begin{thm}
\label{nonsynchgeom}
The following graphs are non-synchronizing:
\begin{enumerate}
\item where $q$ is a prime-power,
the line graph of $\mathrm{PG}(3,q)$;
\item
where $q$ is a prime-power,
the point graph of ${\cal H}(3,q^2)$; 
\item
where $q$ is a prime-power, the graph whose vertices are
the points of $\mathrm{PG}(2,q^2)$ not in ${\cal H}(2,q^2)$, with
distinct vertices $v,w$ joined by an edge precisely when the line of
$\mathrm{PG}(2,q)$ through $v$ and $w$ meets ${\cal H}(2,q^2)$ in just
one point (i.e{.} is a tangent line);
\item
where $q$ is a power of 2,
the complement of the point graph of $W(3,q)$. 
\end{enumerate}
\end{thm}

\begin{proof}
\begin{enumerate}
\item
The projective space $\mathrm{PG}(3,q)$ has a \deftext{resolution}
(also called a \deftext{packing} or \deftext{parallelism}), that is,
a partition of the lines into spreads, where a \deftext{spread} is a
set $\cal S$ of lines such that every point is incident with exactly
one line in $\cal S$. See \cite{Den}.

Thus, the line graph of $\mathrm{PG}(3,q)$ (which is neither null nor
complete) has a proper vertex colouring whose colour classes are the
spreads in a resolution, and has a clique consisting of the lines incident
with a given point, which has a vertex in each colour class.
\item
This follows from the fact that ${\cal H}(3,q^2)$ has a \deftext{fan},
that is, a partition of its point-set into ovoids, where an
\deftext{ovoid} is a set $\cal O$ of points such that every line is
incident with exactly one point in $\cal O$. See \cite[Theorem~6.8]{ACS},
and \cite{BW} or \cite[Proposition~2.7.3]{BVM}.
\item
This result of the author was inspired by the present work. Its proof,
that $\chi(\Delta)=\omega(\Delta)=q^2$,
now appears as a small part of \cite[Section~3.1.6]{BVM}, 
where the graph $\Delta$ is denoted by $NU_3(q)$.
\item
This follows from the fact that, for $q$ even, 
$W(3,q)$ has both an ovoid and a spread. See \cite[Section~6.2]{ACS}.
\end{enumerate} 
\end{proof}

\subsection{Some non-synchronizing graphs with a non-trivial group
fixing a set of colour classes}

In certain difficult cases, we were able to find 
a $G$-invariant graph 
$\Gamma$ that has a vertex $\omega(\Gamma)$-colouring whose set
of colour-classes is invariant under a non-trivial
subgroup $H$ of $G$. Here, the key \textsf{GRAPE} \cite{GRAPE} function used
was \texttt{GRAPE\_ExactSetCover}, which, given a permutation group $G$
on $\{1,\ldots,n\}$, certain subsets of $\{1,\ldots,n\}$ and $H\le G$,
either finds an $H$-invariant exact set cover
of $\{1,\ldots,n\}$ consisting of elements in the $G$-orbits of the
given subsets, or shows that no such set cover exists. 
Here, the given subsets would be $G$-orbit representatives
of the cocliques of $\Gamma$ of maximum size. 

\begin{table}
\begin{center}
\begin{tabular}{|r|r|l|r|r|l|}
\hline
$n$ & lib nr & $G=\mathrm{Aut}(\Gamma)$ & $\mathrm{deg}(\Gamma)$ & $\omega(\Gamma)$ & $H$ \\
\hline
248 & 1 & PSL(2, 31) & 150 & 31 & $C_{31}$ \\
280 & 22 & J\_2.2 & 135 & 28 & $C_7$ \\
315 & 1 & PSp(6, 2) & 152 & 45 & $C_3\times C_3$ \\
462 & 2 & Sym(12) & 425 & 66 & $C_{11}$ \\
495 & 4 & M(12).2 & 472 & 165 & $C_{11}$ \\
495 & 6 & M(12).2 & 398 & 55 & $C_{11}$ \\
560 & 2 & PSz(8).3 & 507 & 112 & $C_7$ \\
574 & 3 & PSL(2, 41) & 532 & 82 & $C_{41}$ \\
620 & 1 & PSL(2, 31) & 595 & 155 & $C_{31}$ \\
\hline
\end{tabular}
\end{center}
\caption{Some non-synchronizing graphs $\Gamma$ with a minimum vertex-colouring
whose set of colour-classes is $H$-invariant}
\label{Hinvariant}
\end{table}

We give in Table~\ref{Hinvariant} a summary of the non-synchronizing
graphs found this way. The ``lib~nr" column gives the \textsf{GAP}
\cite{GAP} Primitive Groups Library number for 
a primitive group $G=\mathrm{Aut}(\Gamma)$ of degree $n$ 
of a graph $\Gamma$ with vertex-degree
$\mathrm{deg}(\Gamma)$, clique number $\omega(\Gamma)$, and having a
vertex $\omega(\Gamma)$-colouring such that the set of colour
classes is $H$-invariant. The group $H$ is a cyclic Sylow subgroup of
$G$, except for $H\cong C_3\times C_3$, which is the derived subgroup
of a Sylow 3-subgroup.

In certain other difficult cases, the key was only to search for
a vertex $\omega(\Gamma)$-colouring whose colour-classes are fixed class-wise
by some subgroup $K$ of $G$, found by working through the cyclic groups
generated by $G$-conjugacy class representatives of a given order.

We give in Table~\ref{Kinvariant} a summary of the non-synchronizing
graphs found this way. We remark that the graph on 175 vertices is
the edge-graph of the Hoffman-Singleton graph. That this edge-graph
has chromatic number equal to its clique number also follows from the
fact that the Hoffman-Singleton graph has chromatic index~7, which
had been determined by Gordon Royle (and others \cite{CGH}) previously.
The graph on 525 vertices is the line graph of the ``Cohen-Tits near
octagon". That this graph has chromatic number equal to its clique
number also follows from the fact that the Cohen-Tits near octagon has a
resolution (see \cite{Soi1}), a computation motivated by the present work.

\begin{table}
\begin{center}
\begin{tabular}{|r|r|l|r|r|l|}
\hline
$n$ & lib nr & $G=\mathrm{Aut}(\Gamma)$ & $\mathrm{deg}(\Gamma)$ & $\omega(\Gamma)$ & $K$ \\
\hline
175 & 4 & PSigmaU(3, 5) & 12 & 7 & $C_5$ \\
525 & 6 & J\_2.2 & 12 & 5 & $C_5$ \\
567 & 5 & PSU(4, 3).2\^{ }2 & 120 & 21 & $C_6$ \\
\hline
\end{tabular}
\end{center}
\caption{Some non-synchroning graphs $\Gamma$ with a minimum vertex-colouring
whose colour-classes are class-wise $K$-invariant}
\label{Kinvariant}
\end{table}

\section{The groups of prime, prime-squared or prime-cubed degree}

Every primitive group of prime degree is separating, and hence
synchronizing. 

The primitive groups of prime-squared or prime-cubed degree are largely
handled by applying the following result. 

\begin{thm}
\label{TH:smalldim}
Suppose $p$ is a prime, $G$ is a primitive group of affine type
on $V:=GF(p)^d$, with $d>1$, and $H$ is the stabilizer in $G$ of the zero 
vector in $V$. Then:
\begin{enumerate}
\item
if $H$ is transitive on the 1-dimensional subspaces 
of $V$ then $G$ is separating;
\item
if there is a $(d-1)$-dimensional subspace $W$ of $V$ 
and an $H$-orbit ${\cal O}$ of 
1-dimensional subspaces of $V$ such that $W$ contains 
no element of ${\cal O}$ then $G$ is non-synchronizing.
\end{enumerate}
\end{thm}

\begin{proof}
\begin{enumerate}
\item
See \cite[Theorem~7.8]{ACS}.
\item
This is an observation of Peter Cameron. 

Suppose $W$ is a $(d-1)$-dimensional subspace of $V$ and there is an
$H$-orbit ${\cal O}$ of 1-dimensional subspaces of $V$ such that $W$
contains no element of ${\cal O}$. Let $\ell\in{\cal O}$ and $g\in G$.
Then $g=ht$ for some $h\in H$ and some $t$ in the translation
subgroup of $G$. Now $\ell^h$ is not contained in $W$, so
$\ell^h$ is a transversal for the cosets of $W$ in $V$, but then so is its
translate by $t$. Thus, every element in the $G$-orbit of $\ell$
is a transveral for the cosets of $W$ in $V$.  Now construct the graph
$\Gamma$ whose vertices are the elements of $V$, with distinct vertices
$a,b$ joined by an edge if and only if $a$ and $b$ are contained in some
element of the $G$-orbit of $\ell$. Then $G\le \mathrm{Aut}(\Gamma)$,
the cosets of $W$ in $V$ are the colour classes of a vertex $p$-colouring
of $\Gamma$, and $\ell$ is a clique of size $p$ of $\Gamma$.
The result follows.
\end{enumerate} 
\end{proof}

\section{The main computation}

We have written a \textsf{GAP} \cite{GAP} program, making use of the
\textsf{GRAPE} \cite{GRAPE}, \textsf{FinInG} \cite{FinInG}, \textsf{AGT}
\cite{AGT}, and \textsf{DESIGN} \cite{DESIGN} packages to determine the
non-separating primitive groups having degrees in a given set, and of
these non-separating groups, the non-synchronizing ones.  This program
is available publicly, together with a log-file of a run of the program
applied to the degrees in $\{2,\ldots,624\}$ \cite{nonsep}. The program
uses the main approach detailed below.  In addition, for some groups
$G$, the program contains a hint (such as the vertex-degree) of a known
non-synchronizing graph for $G$, or invokes a known result, to avoid
much fruitless computation.  Moreover, in three specific cases, the
program relies on external high-performance parallel computation. This is
documented in the program, and these parallel computations are described
in the next section.

We maintain a list $\cal L$ of non-synchronizing graphs.  This list is
initialized with those graphs described in Theorems \ref{nonsynchcombi}
and \ref{nonsynchgeom} that have degree at most 624 and a vertex-primitive
automorphism group, and the graphs given in Tables \ref{Hinvariant}
and \ref{Kinvariant}. We consider the primitive groups of a given degree in 
order of non-increasing size.

Suppose now $G$ is a primitive group on $\Omega:=\{1,\ldots,n\}$, obtained
from the \textsf{GAP} Library of Primitive Groups, and
we want to determine whether $G$ is non-separating, and if so,
whether $G$ is non-synchronizing.
The general approach is as follows.

If $G$ is of prime degree or $G$ is 2-transitive
then $G$ is separating. 

If $G$ is of affine type of degree $p^2$ or $p^3$, for some prime $p$,
then we apply the tests of Theorem~\ref{TH:smalldim}. 

If $G$ is not taken care of by the previous steps, we compute the non-null
non-complete graphs $\Gamma$ with vertex-set $\Omega$, such that $G\le
\mathrm{Aut}(\Gamma)$.  These are the non-null non-complete ``generalized
orbital graphs" for $G$, and are very efficiently constructed and stored
making use of the \textsf{GRAPE} function \texttt{GeneralizedOrbitalGraphs} 
and the \textsf{GRAPE} package graph data structure \cite{GRAPE}. There
are, however, $2^m-2$ such graphs, where $m$ is the number of
$G$-orbits on the 2-subsets of $\Omega$.
If $m=1$ then $G$ is 2-set transitive, so $G$ is separating and we are 
done with $G$.

Otherwise, it is then checked (using \textsf{nauty} \cite{nautytraces} 
via \textsf{GRAPE}) whether any non-null non-complete generalized 
orbital graph for $G$ is
isomorphic to a graph on the current list ${\cal L}$ of non-synchronizing
graphs, and if so, then $G$ is non-synchronizing and we are done with $G$.

Otherwise, we consider the distinct pairs $\{\Gamma,\bar{\Gamma}\}$
of complementary non-null non-complete generalized orbital graphs for $G$, 
such that the
vertex-degree of $\Gamma$ is less than or equal to that of $\bar{\Gamma}$.
The powerful clique machinery in \textsf{GRAPE}, which exploits graph
symmetry, can usually be used to determine $\omega:=\omega(\Gamma)$, and if
this number divides $n$, then whether there is a clique in $\bar{\Gamma}$
of size $n/\omega$.  We make use of the \textsf{AGT} package to check
first whether the Hoffman coclique bound (also known as the ratio bound)
for $\Gamma$ is at least $n/\omega$.  (For some difficult cases a clique
of $\bar{\Gamma}$ of size $n/\omega$ is found by assuming an appropriate
$H\le G$ stabilizing the clique.)  If $\bar{\Gamma}$ has a clique of
size $n/\omega$ (equivalently, $\Gamma$ has a coclique of this size),
then $G$ is non-separating.  In this case, if $G$ is not almost simple
then $G$ is non-synchronizing and we are done with $G$.

On the other hand, if every distinct pair $\{\Gamma,\bar{\Gamma}\}$
of non-null non-complete generalized orbital graphs for $G$ has
$\omega(\Gamma)\omega(\bar{\Gamma})<n$, then we know that $G$ is separating,
and hence synchronizing.

We are left with the case where our primitive $G$ is almost simple and
non-separating.  We must examine all the non-separating graphs for $G$
to determine whether any of them is non-synchronizing.  For that, the
proper vertex-colouring machinery in \textsf{GRAPE}, which exploits the
automorphism group of a graph, can usually be applied successfully.
If a non-synchronizing graph is found for $G$, then of course $G$ is
non-synchronizing, and we add this graph to our list ${\cal L}$. If
no non-separating graph for $G$ is non-synchronizing, then $G$ is
synchronizing.

The main computation found, and added to the list ${\cal L}$, the 
non-synchronizing graphs given in Table~\ref{nonsynchfurther}. 
This completed the proof of the following:

\begin{thm}
Let $G$ be a non-synchronizing primitive almost simple group of degree at
most 624. Then $G$ is a group of automorphisms of a graph isomorphic
to one of those described in Theorem~\ref{nonsynchcombi},
Theorem~\ref{nonsynchgeom}, Table~\ref{Hinvariant},
Table~\ref{Kinvariant}, or Table~\ref{nonsynchfurther}.  
\end{thm}

\begin{table}
\begin{center}
\begin{tabular}{|r|r|l|r|r|l|}
\hline
$n$ & lib nr & $G=\mathrm{Aut}(\Gamma)$ & $\mathrm{deg}(\Gamma)$ & $\omega(\Gamma)$ \\
\hline
21 & 1 & PGL(2, 7) & 4 & 3 \\
52 & 1 & PSL(3, 3).2 & 6 & 4 \\
55 & 3 & PGL(2, 11) & 8 & 5 \\
56 & 7 & Sym(8) & 25 & 8 \\
91 & 4 & PGL(2, 13) & 48 & 13 \\
102 & 1 & PSL(2, 17) & 95 & 34 \\
105 & 6 & PSL(3, 4).D\_12 & 8 & 5 \\
117 & 1 & PSL(3, 3).2 & 16 & 9 \\
120 & 11 & Sym(8) & 14 & 8 \\
126 & 13 & Sym(9) & 60 & 9 \\
135 & 3 & PSO+(8, 2) & 64 & 9 \\
144 & 4 & M(12).2 & 66 & 12 \\
186 & 1 & PSL(3, 5).2 & 10 & 6 \\
234 & 2 & PSL(3, 3).2 & 96 & 13 \\
285 & 1 & PGL(2, 19) & 156 & 19 \\
336 & 7 & PSp(6, 2) & 225 & 28 \\
336 & 6 & PSL(3, 4).D\_12 & 30 & 16 \\
364 & 4 & PSO(7, 3) & 243 & 28 \\
396 & 2 & M(12).2 & 200 & 11 \\
425 & 1 & PSp(4, 4).4 & 256 & 17 \\
456 & 2 & PSL(3, 7).Sym(3) & 441 & 57 \\
465 & 3 & PSL(5, 2).2 & 28 & 15 \\
496 & 8 & PSL(5, 2).2 & 30 & 16 \\
520 & 7 & PSL(4, 3).2\^{ }2 & 24 & 13 \\
\hline
\end{tabular}
\end{center}
\caption{Further non-synchronizing graphs}
\label{nonsynchfurther}
\end{table}

\section{Parallel computing}

In just three cases, we made use of a hybrid
\textsf{GAP}/\textsf{GRAPE}/C program of the author for computing the
cliques with given vertex-weight sum in a graph whose vertices are
weighted with non-zero $d$-vectors of non-negative integers.
\textsf{GAP} and \textsf{GRAPE} are used to exploit graph symmetry and
to build the top of the search tree. Then, a C program
is used on the Apocrita HPC cluster \cite{Apocrita} to handle the 
lower branches of the search tree in parallel. 
This set-up has been very useful in the classification of combinatorial
designs (see \cite{BS}) and the calculation of clique numbers.

The most difficult case handled for the current project concerned
the generalized orbital graph $\Gamma$ with vertex-degree 250, clique
number 63, and coclique number 5 for the primitive action of J\_2.2 on
315 points.  The problem was to determine whether the vertices of $\Gamma$
could be partitioned into 63 of its 1008 cocliques of size 5.
For this, a graph $\Upsilon$ was constructed on these 1008 cocliques,
with two cocliques joined by an edge exactly when they have trivial
intersection. Then the vertices of $\Upsilon$ were weighted with their
characteristic vectors (of length 315), and the question was whether
$\Upsilon$ had a clique whose vertex-weights summed to the all-1 vector.
The answer (no) took about 33 hours runtime using 100 cores, and showed
that the graph $\Gamma$ has chromatic number greater than its clique
number. This was used in the proof that the degree 315
primitive permutation representations of J\_2 and J\_2.2 are synchronizing
(but non-separating).

A similar parallel computation, using about 140 minutes runtime on 100
cores on the Apocrita cluster was used to show that the generalized
orbital graph with vertex-degree 144, clique number 25, and coclique
number 13 for the primitive action of PSp(4,~5).2 on 325 points does not
have a vertex 25-colouring. This was used in the proof that the 
degree 325 primitive permutation representations of PSp(4,~5)
and PSp(4,~5).2 are synchronizing (but non-separating).

A further separate computation was performed to handle the 262143
complementary pairs of non-null non-complete 
generalized orbital graphs for the primitive degree
465 representation of PSL(2,~31), to show that this group and its overgroup
PGL(2,~31) on 465 points are both separating.  This was mostly done
using \textsf{GAP}, \textsf{GRAPE}, and \textsf{AGT}, but for two pairs
$\{\Gamma,\bar{\Gamma}\}$ with $\Gamma$ having vertex-degree 24 and clique
number 5, a parallel computation using the \textsf{GAP}/\textsf{GRAPE}/C
program of the author on the Apocrita cluster 
was used to show that $\omega(\bar{\Gamma})<93$.

\section{The synchronizing groups}

It turns out that, amongst the 4538 primitive permutation groups having
degree in $\{2,\ldots,624\}$, precisely 3298 are synchronizing. Of these,
just 348 are neither 2-transitive nor have prime degree, and of these,
just 19 are non-separating. We list these non-separating synchronizing
groups in Table~\ref{nonsepsynch}.

The groups of degrees 40, 156, and 400 in Table~\ref{nonsepsynch}
belong to the only known infinite sequences of non-separating
synchronizing groups. They arise via their actions as orthogonal
groups on the ``parabolic quadric'' $Q(4,p)$ with $p$ an odd prime (see
\cite[Example~6.9]{ACS}).  The groups of degree 63 and 210 are also known
examples of non-separating synchronizing groups. Those of degree 63 come
from the primitive actions of degree 63 of PSU(3,~3).2 and its socle
PSU(3,~3), such that this socle has permutation rank~5. Those of degree 210
are discussed in \cite{ABC}. The remaining examples appear to be new.

\begin{table}
\begin{center}
\begin{tabular}{|r|r|l|}
\hline
$n$ & lib nr & $G$ \\
\hline
40 & 4 & PSp(4, 3):2 \\
40 & 3 & PSp(4, 3) \\
63 & 4 & PSU(3, 3).2 \\
63 & 3 & PSU(3, 3) \\
156 & 4 & PSp(4, 5).2 \\
156 & 3 & PSp(4, 5) \\
210 & 2 & Sym(10) \\
210 & 1 & Alt(10) \\
280 & 12 & PSL(3, 4).D\_12 \\
280 & 9 & PSL(3, 4).Sym(3) \\
280 & 10 & PSL(3, 4).Sym(3) \\
280 & 11 & PSL(3, 4).6 \\
280 & 7 & PSL(3, 4).3 \\
315 & 3 & J\_2.2 \\
315 & 2 & J\_2 \\
325 & 2 & PSp(4, 5).2 \\
325 & 1 & PSp(4, 5) \\
400 & 4 & PSp(4, 7):2 \\
400 & 3 & PSp(4, 7) \\
\hline
\end{tabular}
\end{center}
\caption{The non-separating synchronizing groups $G$ of degree up to 624}
\label{nonsepsynch}
\end{table}

Although there are synchronizing groups of diagonal type
(see \cite{BGLR1}), there are no such groups of degree up to 624.

\section*{Acknowledgements} 

This research utilised Queen Mary's Apocrita HPC facility \cite{Apocrita},
supported by QMUL Research-IT. I thank Peter Cameron for
interesting and useful discussions about synchronizing and related groups.
I am grateful to Richard Parker for many discussions and insights
on mathematics and computing over many years.

\end{document}